\pdfoutput=1
\PassOptionsToPackage{unicode}{hyperref}
\documentclass[11pt,a4paper]{amsart}

\usepackage[utf8]{inputenc}
\usepackage[active]{srcltx}
\usepackage{microtype}
\usepackage{amsfonts,enumerate}
\usepackage[mathscr]{euscript}
\usepackage{epsfig}
\usepackage{latexsym}
\usepackage[all]{xy}
\usepackage{amssymb}
\usepackage{amsthm}
\usepackage{amsmath}
\usepackage{amsxtra}
\usepackage{xcolor}
\usepackage[OT2,T1]{fontenc}
\usepackage{url}
\usepackage{hyperref}
\usepackage{stackengine}

.tfm
.tfm

\numberwithin{equation}{section}

\theoremstyle{plain}
\newtheorem{prop}{Proposition}[section]
\newtheorem{thm}[prop]{Theorem}
\newtheorem*{mainthm}{Theorem~\ref{thm:bound}}
\newtheorem*{corotwist}{Corollary~\ref{coro:twists}}
\newtheorem{coro}[prop]{Corollary}

\newtheorem{lemma}[prop]{Lemma}

\newtheorem{hypos}[prop]{Hypotheses}

\theoremstyle{definition}
\newtheorem{defi}[prop]{Definition}

\newtheorem{remark}[prop]{Remark}
\newtheorem{exm}{Example}

\theoremstyle{remark}

\numberwithin{table}{section}

\DeclareMathOperator{\rec}{rec}
\DeclareMathOperator{\Res}{Res}

\DeclareMathOperator{\res}{res}

\DeclareMathOperator{\Gal}{Gal}
\DeclareMathOperator{\Cl}{Cl}

\DeclareMathOperator{\Hom}{Hom}
\DeclareMathOperator{\Frac}{Frac}

\DeclareMathOperator{\im}{Im}

\DeclareMathOperator{\Char}{char}

\DeclareMathOperator{\sign}{sign}

\newcommand{\cO}{\mathcal{O}}

\newcommand{\Om}{{\mathscr{O}}}

\newcommand{\Disc}{\Delta}

\newcommand{\Sel}{\mathrm Sel}
\newcommand{\Asq}{(A_K^\times/(A_K^\times)^2)_\square}
\newcommand{\A}{A_K^\times/(A_K^\times)^2}
\newcommand{\AsqO}{(A_\Om^\times/(A_\Om^\times)^2)_\square}

\newcommand{\Cll}{\Cl_*}

\def\ZZ{\mathbb Z}

\def\RR{\mathbb R}
\def\FF{\mathbb F}
\def\QQ{\mathbb Q}

\def\CC{\mathbb C}
\def\C{\mathscr C}

\def\Ml{C_*}
\def\Mu{\tilde{C}}

\def\nsquare{\ensurestackMath{\stackinset{c}{}{c}{}{/}{\square}}}

\def\<#1>{{\left\langle{#1}\right\rangle}}

\def\Z{{\mathbb Z}}             
\def\Q{{\mathbb Q}}             

\def\set#1{{\left\{{\def\st{\;:\;}#1}\right\}}}

\def\id#1{{\mathfrak{#1}}}      

\DeclareMathOperator{\norm}{{\mathscr N}}

\DeclareMathOperator{\trace}{{\mathrm{Tr}}}

\def\dkro#1#2{\left(\frac{#1}{#2}\right)}             
\let\kro\dkro

\newcommand{\lmfdbnumberfield}[1]{\href{http://www.lmfdb.org/NumberField/#1}{#1}}
\newcommand{\lmfdbec}[3]{\href{http://www.lmfdb.org/EllipticCurve/Q/#1/#2/#3}{{\text{\rm#1.#2#3}}}}

\DeclareSymbolFont{cyrletters}{OT2}{wncyr}{m}{n}
\DeclareMathSymbol{\Sha}{\mathalpha}{cyrletters}{"58}

\newcommand{\hooklongrightarrow}{\lhook\joinrel\longrightarrow}

\title[$2$-Selmer groups and quadratic twists]{On $2$-Selmer groups and quadratic twists of elliptic curves}

\author{Daniel Barrera Salazar}
\address{DMCC \\ Universidad de Santiago de Chile \\ Alameda 3363, Santiago, Chile.}
\email{danielbarreras@hotmail.com}
\thanks{DBS was supported by the FONDECYT PAI 77180007}

\author{Ariel Pacetti}
\address{FaMAF-CIEM (CONICET) \\ Universidad Nacional de C\'ordoba \\
Medina A\-llen\-de s/n, Ciudad Universitaria, 5000 C\' ordoba, Rep\'
ublica Argentina.}
\email{apacetti@famaf.unc.edu.ar}
\thanks{AP was partially supported by PICT-2018-02073 and Proyecto Consolidar 2018-2021 33620180100781CB}

\author{Gonzalo Tornar{\'\i}a}
\address{Universidad de la Rep\'ublica \\ Montevideo, Uruguay}
\email{tornaria@cmat.edu.uy}
\thanks{GT was partially supported by ANII--FCE 2017--136609}
\dedicatory{To the memory of John Tate}

\begin{document}

\begin{abstract}
  Let $K$ be a number field and $E/K$ be an elliptic curve with no
  $2$-torsion points. In the present article we give lower and upper
  bounds for the $2$-Selmer rank of $E$ in terms of the $2$-torsion of a
  narrow class group of a certain cubic extension of $K$ attached to $E$.
  As an
  application, we prove (under mild hypotheses) that a positive
  proportion of prime conductor quadratic twists of $E$ have the same
  $2$-Selmer group.
\end{abstract}

\maketitle

\section*{Introduction}

Given an elliptic curve $E$ over a number field $K$, the Mordell-Weil
theorem implies that its set of $K$-rational points is a finitely
generated abelian group. In particular, it has a torsion part and a
free one. From a computational point of view finding the torsion
part is the ``easy'' task (and is implemented in most number theory
computational systems, such as PARI/GP~\cite{PARI2}, SageMath~\cite{sagemath} or
Magma~\cite{Magma}). The computation of the free part is more subtle, and
involves the \emph{descent} method (see for example section VIII.3 of
\cite{MR2514094}), and is still an open question whether the proposed
algorithms to compute ranks of elliptic curves end or not (depending
on the finiteness of the Tate-Shafarevich group).

The most effective way to compute the rank is to apply $2$-descent,
which involves computing the $2$-Selmer group (see
Definition~\ref{defi:2selmer}). Since computing the $2$-Selmer group
involves hard computations, a natural question is whether one can give
a bound for it. Let $E$ be an elliptic curve of the
  form $$E:y^2=F(x)$$ with $F(x) \in \Om_K[x]$ a monic irreducible cubic
  polynomial and let $A_K = K[x]/F(x)$ be the cubic extension of
  $K$ given by $F(x)$.
In \cite{Brumer} the authors give, for $K=\QQ$, an upper bound for
semistable elliptic curves in terms of the class group of $A_\QQ$
(which can be efficiently computed), and is the first article to show
a relation between the $2$-Selmer group and a class group. Recently,
in \cite{MR3934463} the author used similar ideas to provide a lower
bound for the $2$-Selmer group of a rational elliptic curve under some
restrictive hypotheses;
namely has  odd and square-free discriminant.

The purpose of
the present article is to extend the ideas of Brumer-Kramer and Li to
give both lower and upper bounds for the $2$-Selmer groups of elliptic
curves over number fields with more relaxed hypotheses.
Denote by $\Cll(A_K,E)$
the narrow class group given in Definition~\ref{defi:Cl*}.
One of our main results is the following:
\begin{mainthm}
  Let $K$ be a number field and let $E/K$ be an elliptic curve satisfying hypotheses~\ref{hypo:KF}. Then
  \[
    \dim_{\FF_2}\Cll(A_K,E)[2] \le \dim_{\FF_2}\Sel_2(E) \le \dim_{\FF_2}\Cll(A_K,E)[2]+{[K:\Q]}.
  \]
  In particular, if $K=\QQ$, the order of the Selmer group is determined by the $2$-torsion of $\Cll(A_K,E)$ and the root number of $E$.
\end{mainthm}

The advantage of our results is twofold: on the one hand, we can get a
lower and upper bound which we expect to be sharp for general number
fields (we show this is the case in some examples).

On the other hand, our relaxed hypotheses allow us to consider
families of quadratic twists of elliptic curves: let for example
$E/\QQ$ be a rational elliptic curve satisfying
hypotheses~\ref{hypo:KF}, and let $p$ be a prime congruent to $1$
modulo $4$ which is inert or totally ramified in $A_K$. Then the
quadratic twist $E_p$ of $E$ by a character of conductor $p$ also
satisfies hypotheses \ref{hypo:KF}, hence our rank bound also applies
to its $2$-Selmer group. Studying the root number change from $E$ to
$E_p$ allows us to deduce that for a positive proportion of them, the
rank of their $2$-Selmer group is constant. In particular, all such
twists have precisely the same $2$-Selmer group (see
Theorems~\ref{thm:negativetwists} and ~\ref{thm:positivetwists}). Such
interesting phenomena has implication in distributions of ranks of
elliptic curves under quadratic twists and the order of the
Tate-Shafarevich group $\Sha(E_p)[2]$. For example, let $d_2(E)$ denote the $2$-Selmer rank of $E$ and define
\[
  N_r(E,X) = |\{\text{quadratic } L/\QQ \; : \; d_2(E^L) = r \text{ and }|\delta(L/\QQ)|<X\}|,
  \]
  where $E^L$ denotes the quadratic twist of $E$ corresponding to $L$ and $\delta(L/\QQ)$ is the discriminant of the extension $L/\QQ$. A direct application of our
  result proves the following Corollary.
  \begin{corotwist}
    Let $E/\QQ$ be an elliptic curve satisfying
    hypotheses~\ref{hypo:KF}, and suppose furthermore that either
    $\Disc(E)<0$ or $\Cl_+(A_\QQ)=\Cl(A_\QQ)$. Let $r \ge 0$, and
    suppose that $E$ has a quadratic twist by a prime inert in $A_\QQ$
    whose $2$-Selmer group has rank $r$.  Then
    $N_r(E,X) \gg X/\log(X)^{1-\alpha}$, where
    \[
      \alpha = \begin{cases}
        1/3 & \text{if }A_\QQ/\QQ \text{ is Galois},\\
        1/6 & \text{otherwise.}
        \end{cases}
\]
      \end{corotwist}
      When $\Disc(E)>0$ a similar result holds replacing $\alpha$ by
      $\alpha/2$.  Such results are important to understand the so
      called Goldfeld's conjecture.  In \cite{MR2660452} the authors
      study the problem of the variation of the $2$-Selmer group in
      quadratic twists families, and they obtain a little stronger result
      for any base field $K$ (see Theorem 1.4), although their
      techniques are slightly different from ours. In \cite{MR3954912}  the authors obtain similar results as ours over $\QQ$ (see the proofs of \cite[Lemma 5.9 and Lemma 5.10]{MR3954912} and  \cite[Theorem 1.12]{MR3954912}).
      
      An immediate application of the result is the following: suppose that $E/\QQ$
      is an elliptic curve with trivial $2$-Selmer group, and let
      $K/\QQ$ be the (infinite) polyquadratic extension obtained by
      composing all quadratic extensions in the hypothesis of
      Theorems~\ref{thm:negativetwists} and
      ~\ref{thm:positivetwists}. Then $E(K)$ is finitely generated
      (see Corollary~\ref{coro:inifiteextension}).

The article is organized as follows: Section 1 contains the local
computations of the Kummer map and its image, which are needed to
bound the $2$-Selmer group. Section 2 contains the main result
(Theorem~\ref{thm:bound}). Our main contributions are: we can work
with polynomials $F(x)$ which do not generate the whole ring of
integers of $A_K$ (a key fact for allowing quadratic twists), and also
we explain in detail how to handle the case of ``positive
discriminants'', i.e. the real places of $K$ where the discriminant of
$F(x)$ is positive. In order to treat this case we work with a
``narrow class group'' instead of a classical one.  Section 3 contains
the application of the main results to families of quadratic
twists. We stated two results (Theorems~\ref{thm:negativetwists} and
\ref{thm:positivetwists}) for elliptic curves over $\Q$ (which
historically received a lot of attention) but they have a similar
version over general number fields. At last, Section 4 includes many
examples of elliptic curves over number fields; the purpose of the examples is to show that the bounds obtained in
this article are sharp for different number fields. At the same time
we show that the lower bound and the upper bound do not hold when
twisting by primes not satisfying hypotheses~\ref{hypo:KF}.

\section{Kummer map}
Let us recall some general statements on the $2$-Selmer group on
elliptic curves (we refer to Section $2$ of \cite{Brumer}).  Let $K$
be a field of characteristic different from $2$, let $\bar{K}$ be a
Galois closure of $K$ and let $G_K=\Gal(\bar{K}/K)$. Let $E/K$ be an
elliptic curve of the form $$E: y^2= F(x)$$ for some monic cubic
square-free
polynomial $F(x) \in K[x]$. The following exact
sequence of $G_K$-modules
\[
    0 \longrightarrow
    E(\bar K)[2] \longrightarrow
    E(\bar K) \xrightarrow{\,\times 2\;}
    E(\bar K) \longrightarrow
    0
\]
  gives rise to a long exact sequence in cohomology. In particular, it induces an injective morphism called the Kummer map
  \[
      \delta_K:E(K)/2E(K) \hooklongrightarrow H^1(G_K,E(\bar{K})[2]).
\]
Let $A_K$ be the $K$-algebra $K[T]/(F(T))$. Then
$H^1(G_K,E(\bar{K})[2])$ is isomorphic to the subgroup of elements in
$A_K^\times/(A_K^\times)^2$ whose norm is a square in $K$ (see
\cite[p. 240]{MR199150}); let us denote by $\Asq$ such set. In
particular, we get an injective map
\[
\delta_K:E(K)/2E(K) \hooklongrightarrow \Asq.
\]
Explicitly, let $P \in E(K)$ and let $x(P)$ denote its first coordinate. Then,
\[
\delta_K(P) = x(P)-T,
\]
whenever $x(P)-T$ is invertible in $A_K$ (see \cite[p. 716-717]{Brumer}).
Note that the algebra $A_K$ and the map $\delta_K$ do not depend on
the choice of model for $E$. Moreover, we denote by $\delta_{K}(E)$
the image of the Kummer map and remark that it is a hard problem to
describe it.

\subsection{The case $K$ a complete archimedean field}
Let $\Delta$ denote the discriminant of $F(x)$ and $K$ be a complete
archimedean field. Then clearly
\[
  A_K \simeq \begin{cases}
    \RR \times \CC & \text{ if } K = \RR \text{ and } \Delta<0,\\
    \RR \times \RR \times \RR & \text{ if } K = \RR \text{ and } \Delta>0,\\
     \CC \times \CC \times \CC & \text{ if } K = \CC.
    \end{cases}
  \]
In the second case,
let $\theta_1 < \theta_2 < \theta_3$ denote the roots of $F(x)$,
and take the isomorphism sending $T$ to $(\theta_1,\theta_2,\theta_3)$.
\begin{lemma}
\label{lemma:imagekummerinfty}
  For complete archimedean fields, the following holds:
  \begin{itemize}
  
  \item[(i)] If $K=\RR$ and $\Delta<0$, $\delta_\RR(E) = \{(1,1)\}$,
    
  \item[(ii)] If $K=\RR$ and $\Delta>0$, $\delta_\RR(E) = \<(1,-1,-1)>$,
  
  \item[(iii)] If $K=\CC$,  $\delta_\CC(E) = \{(1, 1, 1)\}$.
  \end{itemize}
\end{lemma}
\begin{proof}
  In cases (i) and (iii) $E(K)/2E(K)$ is trivial.
  In case (ii) $E(K)/2E(K)$ has order $2$, and a point $P$ with
  $\theta_1 < x(P) < \theta_2 < \theta_3$ maps to $(1,-1,-1)$ up to squares
  (see also \cite[Proposition 3.7]{Brumer}).
\end{proof}

\begin{remark}
\label{rem:distinguishedplace}
  When $K=\RR$ and $\Delta >0$, $A_K$ has three real places, and one of
  them is distinguished, as it is the unique one satisfying that the
  composition of $\delta_\RR$ with its projection is
  trivial. Lemma~\ref{lemma:imagekummerinfty} states that when the
  roots of $F(T)$ are ordered, such place corresponds to the first
  one, but for a general elliptic curve $E/\RR$, we can always talk of such
  distinguished place.
\end{remark}

\subsection{The case $K$ is a finite extension of
\texorpdfstring{$\Q_p$}{ℚₚ}}
For the rest of this section we assume that $K$ is a finite extension of
$\QQ_p$. Let $\Om$ denote its ring of integers, $\id{p}$ its maximal
ideal, $\pi$ a generator of $\id{p}$ and $k=\Om/\id{p}$ its residue field.

\begin{lemma}
\label{lemma:E/2E}
  The order of $\delta_K(E)$ equals $[\Om:2\Om]\cdot \bigl|E(K)[2]\bigr|$.
  \end{lemma}
\begin{proof}
  See Lemma 3.1 of \cite{Brumer}.
\end{proof}
Let $A_\Om$ be the ring of integers of $A_K$.
\begin{remark}
\label{rem:norm}
    Since $A_K$ is isomorphic to a product of local fields,
  $A_\Om$ is isomorphic to the product of the ring of integers of such
  fields. Furthermore, since $[A_K:K]=3$, the norm map
  $\norm: A_\Om^\times/(A_\Om^\times)^2 \to \Om^\times/(\Om^\times)^2$
  is surjective (it is the identity on the class of elements of $\Om^\times$).
\end{remark}

Denote $\AsqO$ the subgroup of elements in
$A_\Om^\times/(A_\Om^\times)^2$ with square norm.
There is a natural inclusion $\AsqO \subset \Asq$.
\begin{lemma}
\label{lemma:Asq}
  The order of $\AsqO$ equals $[\Om:2\Om]^2\cdot\bigl|E(K)[2]\bigr|$.
\end{lemma}
\begin{proof}
  The $K$-algebra $A_K$ is isomorphic to a product of fields
  $L_1 \times \dotsb \times L_t$, where $1\leq t \leq 3$.
  Let $R_i$ be the ring of integers of
  $L_i$, so that $A_\Om \simeq R_1\times\dotsb\times R_t$.
  By \cite[Proposition
  63:9]{MR1754311} we have
  $[\Om^\times:(\Om^\times)^2]=2\,[\Om:2\Om]$
  and 
  $[R_i^\times:(R_i^\times)^2]=2\,[R_i:2R_i]$.

  Since $A_K$ has dimension 3 over $K$, we have
  $\prod \,[R_i:2R_i] = [\Om:2\Om]^3$. It follows that
  $[A_\Om^\times:(A_\Om^\times)^2] = 2^t\,[\Om:2\Om]^3$.
  Since $\norm: A_\Om^\times/(A_\Om^\times)^2 \to \Om^\times/(\Om^\times)^2$
  is surjective, its kernel $\AsqO$ has order
  $[A_\Om^\times:(A_\Om^\times)^2] \mathbin{/} [\Om^\times:(\Om^\times)^2] = 2^{t-1}
  \,[\Om:2\Om]^2$.

  The result follows by noting that $2^{t-1} = \bigl|E(K)[2]\bigr|$.
\end{proof}

\begin{defi}
\label{def:hyp}
    We say that $E:y^2=F(x)$ satisfies (\dag)
    if $F(x)\in\Om[x]$ is a monic cubic square-free polynomial and
    any of the following conditions holds:
    \begin{enumerate}[(\dag.i)]
    \item $A_K$ is a field extension of $K$, or
    \item $A_\Om = \Om[T]/(F(T))$, or
    \item $\Char(k)>2$ and $[E(K):E_0(K)]$ is odd, where $E_0(K)$ is the subgroup of the points of $E(K)$ whose reduction is non-singular, or
    \item $\Char(k)=2$, $K/\QQ_2$ is unramified, and $E$ has good
        reduction.
  \end{enumerate}
\end{defi} \begin{remark} When $\Char(k)>2$ condition
    $(\dagger. i)$ or condition $(\dagger. ii)$ imply
    $(\dagger. iii)$. To see it let
    $I_K \subset \mathrm{Gal}(\overline{K}/K)$ be the inertia subgroup
    and consider the following three possibilities
    $E[2](\overline{K})^{I_K}= \{0\}$ or
    $E[2](\overline{K})^{I_K}\cong \Z/2\Z$ or
    $E[2](\overline{K})^{I_K}\cong (\Z/2\Z)^2$. If
    $E[2](\overline{K})^{I_K}= \{0\}$ then $[E(K):E_0(K)]$ is odd by
    \cite[Lemma 4]{MR2867919}.  Observe that if $(\dagger. i)$ is true
    then $E[2](\overline{K})^{I_K}\cong \Z/2\Z$ is not possible. If
    $E[2](\overline{K})^{I_K}\cong \Z/2\Z$ and $(\dagger. ii)$ holds
    $v_{\mathfrak{p}}(\mathrm{disc}(F(x)))= 1$ hence
    $[E(K): E_0(K)]=1$ by Tate's algorithm (\cite{MR0393039}). Finally if
    $E[2](\overline{K})^{I_K}\cong (\Z/2\Z)^2$ then hypothesis
    $(\dagger. i)$ or $(\dagger. ii)$ implies that $E$ has good
    reduction hence $(\dagger. iii)$ is clearly true.

  \end{remark}
\begin{thm}
\label{thm:inclusion}
  If $E$ satisfies \textup{(\dag)} then $\delta_K(E) \subset \AsqO$.
\end{thm}
\begin{proof} A similar result (for particular cases) is given in Corollary 3.3, Proposition
  3.4, Lemma 3.5, Proposition 3.6, and Lemma 4.2 of \cite{Brumer},
  and in Lemma 2.12 of \cite{MR3934463}.

  Suppose first that $E$ satisfies (\dag.i), i.e.
  $A_K$ is a cubic field extension of $K$ (either
  unramified or totally ramified).  Let
  $\id{P}$ denote the maximal ideal of $A_\Om$. Let
  $P=(x,y) \in E(K)$.  The equality $y^2 = F(x)$ implies that
  $2v_{\id{P}}(y)=3v_{\id{P}}(x-T)$, hence $v_{{\id{P}}}(x-T)=2n$ is
  even. If $\tilde\pi$ denotes a local uniformizer in $A_\Om$ then
  $\tilde\pi^{-2n}(x-T) \in A_\Om^\times$ so that
  $\delta_K(P)\in\AsqO$.


  Suppose now that $E$ satisfies (\dag.ii). If $A_K$ is a field the
  result is already proven, hence we can restrict to the cases
  $A_K \simeq K \times K \times K$ or $A_K \simeq K \times L$, for
  $L/K$ a quadratic extension.  In the case
  $A_K\simeq K\times K\times K$, $F(x)=(x-c_1)(x-c_2)(x-c_3)$
  with $c_i\in\Om$, and hypothesis (\dag.ii) implies $v_\id{p}(c_i-c_j)=0$ for
  $i\neq j$.  Let $P=(x,y)\in E(K)$. If $v_\id{p}(x)<0$ then
  $v_\id{p}(x-c_i)=v_\id{p}(x)$, hence $2v_\id{p}(y)=3v_\id{p}(x)$ and
  so $v_\id{p}(x-c_i)$ is even for all $i$.  Otherwise $v_\id{p}(x)\geq0$
  and $v_\id{p}(x-c_i)\geq0$ for all $i$.  Since
  $v_\id{p}(c_i-c_j)=0$ for $i\neq j$, at least two terms in the right
  hand side of the equality 
  $$2v_\id{p}(y)=v_\id{p}(x-c_1)+v_\id{p}(x-c_2)+v_\id{p}(x-c_3),$$ are
  0,
  hence the third term must also be even.
  In either case 
  $\delta_K(P)=(\pi^{-v_{\id{p}}(x-c_1)}(x-c_1),\pi^{-v_{\id{p}}(x-c_2)}(x-c_2),\pi^{-v_{\id{p}}(x-c_3)}(x-c_3))\in\AsqO$.

  In the case $A_K\simeq K\times L$ for a quadratic
  extension $L/K$ (either unramified or ramified),
  $F(x)=(x-c)(x-\gamma)(x-\gamma')$,
  with $c\in\Om$ and $\gamma\in\Om_L$. Let $v_{\id{p}}$ denote the valuation of $L$ which extends that of $K$. Hypothesis (\dag.ii) implies $v_\id{p}(c-\gamma)=0$
  and $v_\id{p}(x-\gamma)\in\set{0,\frac12}$ if $x\in\Om$.
  Let $P=(x,y)\in E(K)$, then
  $$2v_\id{p}(y)=v_\id{p}(x-c)+2v_\id{p}(x-\gamma).$$
  When $v_\id{p}(x)<0$ this implies
  $v_\id{p}(x-c)=v_\id{p}(x-\gamma)=v_\id{p}(x)$ is even.
  When $v_\id{p}(x)\geq0$ then $v_\id{p}(x-c)\geq0$ and
  $v_\id{p}(x-\gamma)\geq0$, and at least one must be $0$
  since $v_\id{p}(c-\gamma)=0$. It follows that
  $v_\id{p}(x-\gamma)\in\ZZ\cap\set{0,\frac12}$.
  Hence $v_\id{p}(x-\gamma)=0$ and
  $v_\id{p}(x-c)$ is also even.
  In either case we have $v_\id{p}(x-c)$ and $v_\id{p}(x-\gamma)$ are
  both even, hence
  $\delta_K(P)=(\pi^{-v_\id{p}(x-c)}(x-c),
  \pi^{-v_\id{p}(x-\gamma)}(x-\gamma))\in\AsqO$.


  When $E$ satisfies (\dag.iii) the result
  is given in \cite[Corollary 3.3]{Brumer}.
  Finally, suppose $E$ satisfies (\dag.iv).
  When $E$ has supersingular reduction,
  then the assumption $K/\QQ_2$ unramified implies
  that $A_K$ is a cubic ramified extension of $K$ 
  \cite[Proposition 3.4]{Brumer},
  in which case $E$ satisfies (\dag.i) so the result is already proved.
  When $E$ has ordinary reduction,
  the result is proved under the assumption $K/\QQ_2$ unramified
  in \cite[Proposition 3.6]{Brumer}.
\end{proof}

It is not true that Theorem \ref{thm:inclusion} holds in
  full generality. Here are some examples where the hypotheses
  $(\dagger)$ are not satisfied and the statement of Theorem \ref{thm:inclusion} does not hold.

\begin{exm}
  Let $$E:y^2=x(x+3p)(x+1-p)$$ be the elliptic curve over $\QQ_p$,
  with Kodaira type $\text{I}_2$ if $p\neq 2,3$, type $\text{I}_4$ if
  $p=3$ and type $\text{III}$ if $p=2$, and let
  $P=(p,2p)\in E(\QQ_p)$.  Here
  $A_{\QQ_p}\simeq\QQ_p\times\QQ_p\times\QQ_p$, but two roots are
  congruent (hence (\dag.ii) is not satisfied);
  $\delta_{\QQ_p}(P)=(p,4p,1)\notin
  (A^\times_{\Z_p}/(A^\times_{\Z_p})^2)_{\square}$.
\end{exm}
\begin{exm} Let $r \in \ZZ$ with $\kro{r}{p}=-1$
    if $p\neq 2$, 
    $r\equiv1\pmod{8}$ if $p=2$
    and
    let $$E:y^2=x(x^2-rp^2-r^2p^4)$$ be the elliptic curve over $\QQ_p$, with Kodaira type $\text{I}^*_0$
    if $p\neq 2$ and type $\text{I}^*_2$ if $p=2$. Let
  $P=(-rp^2,rp^2)\in E(\QQ_p)$.  Here
  $A_{\QQ_p}\simeq\QQ_p\times\QQ_p(\gamma)$ with
  $\gamma=p\sqrt{r+r^2p^2}$, the latter a quadratic
  unramified extension (whose ring of integers is not generated by
  $\gamma$, so (\dag.ii) is not satisfied); 
  $\delta_{\QQ_p}(P)=(-rp^2,-rp^2-\gamma)\notin(A^\times_{\Z_p}/(A^\times_{\Z_p})^2)_{\square}$.
\end{exm}
\begin{exm} Let
    $$E:y^2=x(x^2-p-p^2),$$ with Kodaira type $\text{III}$ and let $P=(-p,p)\in E(\QQ_p)$.  Here
    $A_{\QQ_p}\simeq\QQ_p\times\QQ_p(\gamma)$ via $T \to (0,\gamma)$ 
    with
    $\gamma=\sqrt{p+p^2}$ generating a quadratic ramified extension. Since the image satisfies that both coordinates are congruent modulo $p$, (\dag.ii) is not satisfied;  $\delta_{\QQ_p}(P)=(-p,-p-\gamma)\notin(A^\times_{\Z_p}/(A^\times_{\Z_p})^2)_{\square}$
\end{exm}



\begin{coro}
\label{coro:kummerimage}
  Suppose that $E$ satisfies \textup{(\dag)} and $\Char(k)>2$.
  Then $\delta_K(E) = \AsqO$.
\end{coro}
\begin{proof}
  By Theorem~\ref{thm:inclusion}
  we know that $\delta_K(E)\subset\AsqO$,
  and by Lemmas~\ref{lemma:E/2E} and \ref{lemma:Asq} both
  sets have the same cardinality.
\end{proof}

\subsubsection{The case $K$ is a finite extension of $\Q_2$}
Consider the set
\[
    U_4 = \{u \in A_\Om^\times \; : \; u \equiv \square \pmod{4 A_\Om} \text{ and
  }\norm(u) = \square\} \subset A_\Om^\times\,.
  \]
  Note that $(A_\Om^\times)^2\subset U_4$.

\begin{lemma}
\label{lemma:squarecondition}
    Suppose $p= 2$. Then:

    \begin{enumerate}
        \item For
    $\alpha\in\Om$ we have
    \[
        1+4\alpha = \square
        \quad
        \Longleftrightarrow
        \quad
        \trace_{k/\FF_2} \alpha = 0
    \]
\item Let $L/K$ be a finite extension with odd ramification index. For
    all $\alpha\in \Om_L$ we have
    \[
        1 + 4\alpha = \square
        \quad
        \Longleftrightarrow
        \quad
        \norm_{L/K}(1+4\alpha) = \square
    \]
\item Let $L/K$ be a finite extension with even ramification index.
    For all $\alpha\in\Om_L$ we have $\norm_{L/K}(1+4\alpha) =
    \square$.
\item The group $\{u\in\Om^\times \;:\; u\equiv\square\pmod{4}\}$
    contains $(\Om^\times)^2$ with index 2.
  \item
\label{part:U4cardinality}
      The index of $(A_\Om^{\times})^2$ in $U_4$ is given by
    \[
    \# (U_4/(A_\Om^{\times})^2) = \begin{cases}
      1 & \text{ if } A_K \text{ is a field},\\
      2 & \text{ if } A_K \simeq K \times L, \text{ with $L$ a field},\\
      4 & \text{ if } A_K \simeq K \times K \times K.
      \end{cases}
    \]
  \end{enumerate}
  \end{lemma}
\begin{proof}
  Note first that if $1+4\alpha$ is a square, say $1+4\alpha = \beta^2$, then
  $\beta \equiv 1 \pmod 2$. Indeed, $v_\id{p}(\beta-1)<v_\id{p}(2)$
  would imply
  $v_\id{p}(\beta+1)=v_\id{p}(\beta-1)<v_\id{p}(2)$, but then
  $v_\id{p}(4\alpha)=v_\id{p}(\beta-1)+v_\id{p}(\beta+1)
  <v_\id{p}(4)$ contradicting $\alpha\in\Om$.


  Furthermore, recall that units in a local field which are congruent
  to $1$ modulo $4\id{p}$ are squares
  (see for example \cite[Theorem 63:1]{MR1754311}),
  hence $1+4\alpha$ is a square in $\cO_L$ if and only
  if there exists $v \in \Om_L$ such that
  $\alpha \equiv v + v^2 \pmod{\id{p}}$ (so $(1+4\alpha) = (1+2v)^2$
  up to squares).

  Consider the map $\phi: \Om/\id{p} \to \Om/\id{p}$ given by
  $\phi(v) = v^2 + v$; it is a group homomorphism with kernel $\{0,1\}$, hence
  its image has index $2$. Furthermore, the composite map
  $\trace_{k/\FF_2} \circ \, \phi : \Om/\id{p} \to \FF_2$ is the
  trivial map. Since the trace map is surjective, we conclude that the
  image of $\phi$ equals the kernel of the trace map, which proves the
  first statement.

  To prove statements $(2)$ and $(3)$, let $L/K$ be a finite extension
  of local fields with ramification index $e_L$, ring of integers
  $\Om_L$, maximal ideal $\id{p}_L$ and
  residue field $k_L$. Clearly
  $\norm_{L/K}(1+4\alpha) \equiv 1+4 \trace_{L/K}(\alpha) \pmod{4
    \id{p}}$, hence the results follow from a comparison between the
  trace map on $L/K$ and the one on their residue fields. Recall that
  if $x \in \Om_L$,
  $\trace_{L/K}(x) \equiv e_L \trace_{k_L/k}(\bar{x}) \pmod{\id{p}}$
  (see \cite{Cassels} Lemma 1, page 20), so the result follows.

  To prove (4) note that $u\equiv\square\pmod{4}$ if and only if
  $K(\sqrt{u})/K$ is a quadratic unramified extension, and there are
  exactly two such extensions (the split extension and the unramified
  field extension).

  The last statement follows easily from the previous ones.
  For example when $A_K$ is a field apply (4) to $A_K$
  to obtain
  $\{u\in A_\Om^\times \;:\;
  u\equiv\square\pmod{4}\}/(A_\Om^\times)^2$
  has exactly two elements and statement (2) implies that only the
  trivial one has square norm. The other two cases follow from a
  similar computation using (3) and (4).
  
\end{proof}
\begin{thm}
\label{thm:2kummerimage}
  Let $K/\QQ_2$ be a finite extension and let $E/K$ be an elliptic
  curve satisfying \textup{(\dag)}.
  Then $\delta_K(E) \subset \AsqO$ with index
  $2^{[K:\QQ_2]}$.  Furthermore, 
  $U_4 \subset \delta_K(E)$.

\end{thm}
\begin{proof}
  The first claim follows from Theorem~\ref{thm:inclusion}
  and Lemmas~\ref{lemma:E/2E} and
  \ref{lemma:Asq}.  For the second statement,
  suppose first that $E$ satisfies (\dag.i), so that $A_K$ is a field.
  Then the result follows since $U_4/(A_\Om^{\times})^2$ is trivial
  by Lemma~\ref{lemma:squarecondition}.

  Suppose now that $E$ satisfies (\dag.ii). The case when $A_K$ is a
  field is already proved.
  Recall that if $E/K$ is given by
  \[
      E : y^2 = x^3 + a_2\,x^2 + a_4\,x + a_6
  \]
  with $a_2,a_4,a_6\in\Om$,
  using the formal group structure on $E$,
  for every $z\in\id{p}$ there is a point $P=P(z)$ with $x$-coordinate
  given by
  \[
      x(P) = z^{-2} - a_2 + O(z^2) \,.
  \]
  Then 
\begin{equation*}
z^2\,\delta_K(P) = z^2\,(x(P)-T) = 1-(a_2+T)\,z^2 +O(z^4)\,.
\end{equation*}

In the case $A_K\simeq K\times K\times K$, by
Lemma~\ref{lemma:squarecondition}, $U_4/(A_\Om^{\times})^2$ has $4$
elements of the form
$\{(\square, \square, \square), (\nsquare, \square, \nsquare),
(\nsquare, \nsquare, \square), (\square, \nsquare, \nsquare)\}$. It is
enough to prove that there exists $z_1, z_2, z_3 \in\Om$ such that
$\delta_K(P(2z_i))$ has $i$-th coordinate not a square. Let
$u=1+4\alpha \in \Om$ be a unit congruent to $1$ modulo $4$ which is not
a square.
Let $\{c_1,c_2,c_3\}$ be the roots of $F(x)$. The hypothesis
(\dag.ii) implies that $\id{p}\nmid (c_2+c_3) = -(a_2+c_1)$,
so there exists
$z_1 \in \Om$ such that
$-(a_2 + c_1) z_1^2 \equiv \alpha \pmod{\id{p}}$ (since the map
$x \to x^2$ is a bijection in $\Om/\id{p}$). Then
$\delta_K(P(2z_1))=(u,\ast,\ast)$
has its first coordinate a non-square. A similar
argument applies to the other two coordinates.

In the case $A_K\simeq K\times L$ for a quadratic extension $L/K$, let
$\{c, \gamma, \gamma'\}$ be the roots of $F(x)$, with
$c \in \Om$ and $\gamma\in\Om_L$.
If $L/K$ is unramified, consider $z=2u$ so the
first coordinate
$z^2\,\delta_K(P(z))_1\equiv 1+4(\gamma+\gamma')u^2 \pmod{4\id{p}}$, and
the hypothesis (\dag.ii) implies that $\gamma+\gamma'\not\in\id{p}$,
hence there exists $u \in \Om$ such that $1+4(\gamma+\gamma')u^2$ is
not a square.  If $L/K$ is ramified, let $\id{P}$ be the maximal ideal
of $\Om_L$
and consider $z=2u$ so the second coordinate
$z^2\delta_K(P(z))_2\equiv 1+4(c+\gamma')u^2\pmod{4\id{p}}$.  The
hypothesis (\dag.ii) implies that $c+\gamma'\not\in\id{P}$ so
there exists $u \in \Om$ (we can take $u \in \Om$ because it has the
same residue field as $\Om_L$) such that $1+4(c+\gamma')u^2$ is not a
square.

In either case, $z^2\delta_K(P(z))\in U_4$ is not a square in $A_K$, but by
Lemma~\ref{lemma:squarecondition}, we know that $U_4/(A_\Om^\times)^2$
has two elements so the statement follows.

  The case (\dag.iii) does not occur since $\Char(k)=2$.
  Finally, suppose $E$ satisfies (\dag.iv). If $E$ has supersingular
  reduction, then $A_K$ is a cubic ramified extension of
  $K$ \cite[Proposition 3.4]{Brumer}, in which case the result is
  already proved. If $E$ has ordinary reduction, the result follows
  from \cite[Proposition 3.6]{Brumer}.
\end{proof}

\section{$2$-Selmer groups and Class groups}\label{s:selmer vs class group}

Suppose now that $K$ is a number field and $E$ is an elliptic curve
over $K$. For $v$ a place of $K$, let $K_v$ denotes its completion.  From
now on we assume the following hypotheses:
\begin{hypos}
\label{hypo:KF}
The elliptic curve $E$ and the field $K$ satisfy:
\begin{enumerate}
    \item The narrow class number of $K$ is odd.
    \item $E(K)[2] = \{0\}$.
    \item For all finite places $v$ of $K$, $E/K_v$ satisfies
        \textup{(\dag)} (see Definition~\ref{def:hyp}).
\end{enumerate}
\end{hypos}
Note that the second hypothesis implies that $A_K$ is a cubic field
extension of $K$ and we denote by $A_{\cO}$ its ring of integers.
For each place $v$ of $K$, let $G_v=G_{K_v}$ and
fix an immersion $G_v \hookrightarrow G_K$.  To ease the notation let
$\delta_v=\delta_{K_v}$ and let $\res_v$  denote the restriction map $H^1(G_K,E(\bar{K})[2]) \rightarrow H^1(G_v,E(\bar{K}_v)[2])$.
\begin{defi}
  \label{defi:2selmer}
  The $2$-Selmer group of $E$ consists of the cohomology classes in
  $H^1(G_K,E(\bar{K})[2])$ whose restriction to $G_v$ lies in the
  image of $\delta_v$ for all places $v$ of $K$, i.e.
  \[
    \Sel_2(E) =\{c \in H^1(G_K,E(\bar{K})[2]) \; : \; \res_v(c) \in \delta_v(E) \ \text{for each place} \ v \ \text{of} \ K\}.
  \]
\end{defi}
%
If $v$ is an archimedean place of $K$ then either:
\begin{enumerate}[(i)]
\item $K_v \simeq \RR$ and $A_{K_v} \simeq \RR \times \CC$,
  
\item $K_v \simeq \RR$ and $A_{K_v} \simeq \RR \times \RR \times \RR$,
\item $K_v \simeq \CC$ and $A_{K_v} \simeq \CC \times \CC \times \CC$.
\end{enumerate}
We say that an archimedean place of $K$ has type (i), (ii) or (iii)
depending on the above cases.  Let us introduce the following
notations: if $\alpha \in A_K$, the notation $\norm(\alpha) \gg 0$
means that for each real archimedean place $v$ of $K$,
$v(\norm(\alpha))>0$.
If $v$ is a real place of $K$ of type (i),
let $\tilde{v}$ denote the unique real place in $A_K$ extending $v$.
If $v$ is a real place of $K$ of type (ii), let
$\tilde{v}$, $\tilde{v}_2$, $\tilde{v}_3$ denote the places of $A_K$
extending $v$, so that $\tilde{v}$ is the distinguished one (see Remark~\ref{rem:distinguishedplace}).

Define the following subgroups of $\A$:
\begin{align*}
    \Ml(E)&= \bigl\{\, [\alpha] \in \A \; : \;
  \text{$A_K(\sqrt{\alpha})/A_K$ is unramified at all finite}
  \notag\\&\hskip 2.2em
  \text{places of $A_K$, it is unramified at $\tilde{v}$
        for all real places $v$ of $K$,}
  \\\notag&\hskip 2.2em
    \text{and for each $v$ of type (ii) it ramifies at $\tilde{v}_2$
        $\Leftrightarrow$ it ramifies at $\tilde{v}_3$}
        \,\bigr\},
\intertext{and}
        \Mu(E)&= \bigl\{\, [\alpha] \in \A \; : \;
           \norm(\alpha) = \square,
                \text{ $w(\alpha)$ is even for all finite}
  \notag\\&\hskip 2.2em
        \text{places $w$ of $A_K$, $\tilde{v}(\alpha)>0$ for all real places $v$ of $K$} \,\bigr\}.
\end{align*}
\begin{remark}
  In the definition of $\Ml(E)$ and $\Mu(E)$, the dependence of $E$
 comes only from the distinguished place in $A_K$ at real archimedean
  places of type (ii). In particular, if no such place exists,
  these subgroups depend only on $A_K$ (and not on the particular
  curve whose cubic field extension of $K$ is $A_K$). 
\end{remark}

\begin{exm}
  \label{exm:sortingroots}
    Here is a concrete example where $\Ml(E)$ depends on $E$ and not
    only on $A_K$:
    consider the
  curve $$E:y^2=F(x)= x^3-7x+3$$ over $K=\QQ$,
  so that $A_\QQ=\QQ[T]/(T^3-7\,T+3)$.
  The real place $v$ of $\QQ$ is of type (ii) since $F$ has three real
  roots $\theta_1<\theta_2<\theta_3$.
  The field $A_\QQ$ has class number $1$, but narrow class number
  $2$. Indeed the narrow Hilbert class field (maximal abelian
  extension unramified at all finite places)
  of $A_\QQ$ is $A_\QQ(\sqrt{T^2-8})$;
  it is unramified at $\tilde{v}$ and
  it is ramified at $\tilde{v}_2$ and $\tilde{v}_3$
  (since $\theta_1^2-8>0$, $\theta_2^2-8<0$ and $\theta_3^2-8<0$).
  Thus $[T^2-8]\in\Ml(E)$ and $\Ml(E)$ has order $2$.

  On the other hand, consider the quadratic twist
  of $E$ by $\QQ(\sqrt{-1})$,
  $$E_{-1}:y^2=-F(-x)=x^3-7x-3\,.$$
  This curve has the same cubic field  $A_\QQ$,
  but twisting changes the distinguished place from $\tilde{v}$ to
  $\tilde{v}_3$, so that $[T^2-8]\not\in\Ml(E_{-1})$ and it follows
  that $\Ml(E_{-1})$ is trivial.
  In Example~\ref{exm:differentclassgroup} we use this to compute the
  2-Selmer rank for all the quadratic twists of this curve.
\end{exm}

\begin{lemma}
\label{lemma:M1}
  The set $\Ml(E)$ equals the set of elements $[\alpha] \in \A$
  satisfying the following local conditions:
  \begin{itemize}
  \item For all finite places $w$ of $A_K$,
      $w(\alpha)$ is even.
  \item For all real places $v$ of $K$, $\tilde{v}(\alpha)>0$.
  \item $\norm(\alpha)\gg 0$.
  \item $\alpha \equiv \square \pmod {4A_{\cO}}$.
  \end{itemize}
\end{lemma}
\begin{proof}
  The only non-trivial part is the condition at places dividing $2$,
  which is a well known result and a detailed proof is given in \cite[Lemma 3.4]{MR3946721}.
\end{proof}
\begin{lemma}\label{lemma:fixed}
We have $\Ml(E)\subset\Asq$.
\end{lemma}
\begin{proof}
    If $[\alpha] \in \Ml(E)$, by Lemma~\ref{lemma:M1} its norm
    $\norm(\alpha)$ has even
 valuation at all finite places of $K$, it is totally positive, and a
 square modulo $4\Om$. Hence $K(\sqrt{\norm(\alpha)})/K$ is unramified
 at all places of $K$, and
 since the class number of $K$ is odd,
 this implies that $\norm(\alpha)$ is a square.
  \end{proof}

\begin{prop}
    \label{prop:selmer inclusion}
    The following inclusions hold
  \[
    \Ml(E) \subset \mathrm{Sel}_2(E) \subset \Mu(E).
    \]
\end{prop}
\begin{proof} 
  Since $\Ml(E)\subset\Asq$, 
  to prove that $\Ml(E) \subset \mathrm{Sel}_2(E)$ it is enough to check
  that if $[\alpha] \in \Ml(E)$ then for each place $v$ of $K$,
  $[\alpha] \in \im(\delta_v)$. The condition at the infinity places is
  clear by Lemma~\ref{lemma:imagekummerinfty}. If $v$ is a finite
  place of $K$ not dividing $2$, as the quadratic
  extension is unramified then $\alpha$ is a unit in $A_{K_v}$
  (up to squares),
  hence by Corollary~\ref{coro:kummerimage} it lies in the image
  of $\delta_v$.  For a place $v$ dividing $2$, by Lemma~\ref{lemma:M1},
  $\alpha \equiv \square \pmod {4A_\Om}$, and by
  Theorem~\ref{thm:2kummerimage} such set is contained in
  the image of $\delta_v$.

The claim $\mathrm{Sel}_2(E) \subset \Mu(E)$ follows from
Lemma~\ref{lemma:imagekummerinfty} and
Theorem~\ref{thm:inclusion}.
\end{proof}

Let $\Frac(A_K)$ denote the group of fractional ideals of $A_K$, let
$P$ be the subgroup of principal ideals,
and consider the subgroup
\begin{align*}
  P_*(E) &= \{ (\alpha) \in P  :  \text{
          and $\tilde{v}_2(\alpha)\,\tilde{v}_3(\alpha)>0$
          for all $v$ of type (ii) }
        \}
\end{align*}

Let $P_+=\{(\alpha) \in P: \alpha\gg 0\}$. Clearly
$P_+ \subset P_*(E) \subset P$ and $P/P_+$ is an elementary
$2$-group.


\begin{lemma} \label{l: other definition of P1}
    We have: 
$$P_*(E)= \{ (\alpha)  \in P:  \text{
            $\tilde{v}(\alpha)>0$ for all real places $v$ of $K$,
            and $\norm(\alpha)\gg0$ }  \}$$
\end{lemma}
\begin{proof}
    The inclusion $\supset$ is trivial.
    For the other inclusion,
    let $(\alpha)\in P_*(E)$. Since the narrow class number of $K$ is
    odd there are units in $K$ with arbitrary signs for the real
    places,
    in particular there is a unit $\mu\in\Om^\times$ such that
    $\tilde{v}(\mu\alpha)>0$ for all real places of $K$. Moreover,
    this implies that $\norm(\mu\alpha)\gg0$.
    Thus $(\alpha)=(\mu\alpha)$ is in the set of the right hand side.
  %
\end{proof}
\begin{defi}
    \label{defi:Cl*}
    Denote by $\Cl(A_K) = \Frac(A_K)/P$ the class group of $A_K$ and
    by $\Cl_+(A_K) = \Frac(A_K)/P_+$ the narrow class group of $A_K$.
    Let $$\Cll(A_K,E) = \Frac(A_K)/P_*(E)$$ denote the class
    group attached to $P_*(E)$.
  \end{defi}
  \begin{remark}
      If $ \Cl_+(A_K)= \Cl(A_K)$, then $P_+ = P = P_*(E)$, 
      therefore
    $\Cll(A_K,E) = \Cl(A_K)$. In particular, 
    $\Cll(A_K,E)$ is independent of the elliptic curve $E$.
  \end{remark}
  
\begin{prop}
\label{prop:C1asclassgroup}
  The group $\Ml(E)$ is isomorphic to the torsion $2$-subgroup of $\Cll(A_K,E)$,
  i.e. $\Ml(E) \simeq \Cll(A_K,E)[2]$.
\end{prop}
\begin{proof}
Let $L$ be the maximal abelian extension of $A_K$ satisfying:
\begin{itemize}
  \item it is unramified at all finite places of $A_K$,
  \item it is unramified at $\tilde{v}$ for all real places $v$ of $K$,
  \item for each $v$ of type (ii), 
      ${G_{\tilde{v}_2}}={G_{\tilde{v}_3}}$
      as subgroups of $\Gal(L/A_K)$.
\end{itemize}
Then $L$ is a finite extension of $A_K$, and
$\Ml(E) \simeq \Hom(\Gal(L/A_K),\mu_2)$.  The Artin reciprocity map
$\rec:\Frac(A_K)\to\Gal(L/A_K)$ has kernel $P_*(E)$, hence
$\Cll(A_K,E)\simeq\Gal(L/A_K)$.  It follows that
$\Ml(E)\simeq\Cll(A_K,E)[2]$ as claimed.
\end{proof}

\begin{thm}
\label{thm:M2order}
  The index $[\Mu(E):\Ml(E)] \le 2^{[K:\QQ]}$.  
\end{thm}

Before giving the proof, we need some auxiliary results. Let $A$, $B$
and $C$ be the set of archimedean places of $K$ of type (i), (ii) and (iii)
respectively, and let $a, b, c$ denote their cardinalities, so
$[K:\Q]=a+b+2c$. Consider the \emph{sign map}
\begin{equation*}
  \sign : A_K^\times \to \prod_{v \in A} \{\pm 1\} \times
          \prod_{v \in B} (\{\pm 1\} \times \{\pm 1\} \times \{\pm 1\})\,.
\end{equation*}
This induces a well defined map on $\A$.
Let
\begin{align*}
    \tilde{W} & = \prod_{v \in A} \{1\} \times \prod_{v \in B} W \\
\intertext{where $W = \{(1,1,1),(1,-1,-1),(-1,1,-1),(-1,-1,1)\}$, and
let}
    \tilde{V} & = \prod_{v \in A} \{1\} \times \prod_{v \in B}V
    \subset\tilde{W}
\end{align*}
where $V = \{(1,1,1),(1,-1,-1)\}$.
Note that
$\sign(\Asq) \subset \tilde{W}$, and that $\sign(\Mu(E))\subset\tilde{V}$.
\begin{lemma}
\label{lemma:directsum}
  There is an isomorphism
  \[
  \sign(\AsqO) \cdot \tilde{V}
  \simeq
  \frac{ \sign(A_\Om^\times)\cdot \tilde{V} }
       { \sign(\Om^\times) }
  \,.
  \]
\end{lemma}
\begin{proof}
  The inclusion $\sign(\AsqO) \subset \sign(A_\Om^\times)$
  induces a morphism
  $\sign(\AsqO) \cdot \tilde{V}
  \to
  ( \sign(A_\Om^\times)\cdot \tilde{V} )/ \sign(\Om^\times)$.
  To prove it is surjective, let
  $\alpha \in A_\Om^\times$. Clearly
  $\norm(\alpha) \in \Om^\times$, so
  $\sign(\alpha) = 
  \sign(\alpha \norm(\alpha))\,
  \sign(\norm(\alpha))$ is the image of
  $\sign(\alpha\norm(\alpha))\in\sign(\AsqO)$.

  To prove it is injective note that
  $\sign(\AsqO)\cdot\tilde{V}\subset\tilde{W}$
  and
  $\sign(\Om^\times)$ satisfies that for
  places $v$ in $B$ its three coordinates are the same,
  hence
  $\sign(\Om^\times)\cap\tilde{W}$ is trivial.
\end{proof}

Let $[\alpha]\in\Mu(E)$. Since $w(\alpha)$ is even for all finite places
$w$ of $A_K$, there is a (unique) ideal $I\in\Frac(A_K)$ such that
$I^2=(\alpha)$. 

\begin{lemma} The association $[\alpha] \mapsto [I]$ induces a well defined map $\phi:\Mu(E)\to\Cl(A_K)$.
\end{lemma}
\begin{proof}
  Let
  $\alpha \in A_K^\times$ and $I \in P$ such that $I^2= (\alpha)$. If $\beta^2 \in (A_K^\times)^2$ then $(\alpha \beta^2) = I^2(\beta^2)= (I\beta)^2$. As $I(\beta)$ lies in the same class as
  $I$ then the map $\phi$ is well defined.
\end{proof}

\begin{proof}[Proof of Theorem~\ref{thm:M2order}] Consider the
  short exact sequences
  \[
      0 \longrightarrow
      \ker\phi \longrightarrow
      \Mu(E) \xrightarrow{\:\:\phi\:\:}
      \Cl(A_K)
  \]
  and
 \[
      0 \longrightarrow
      P/P_*(E) \longrightarrow
      \Cll(A_K,E)[2] \xrightarrow{\:\:\psi\:\:}
      \Cl(A_K)
  \]
    where $\psi$ is the restriction of the natural projection
    $\Cll(A_K,E) \to \Cl(A_K)$.
    From the definition of $\Mu(E)$ and Lemma~\ref{l: other definition of P1}  it
    is clear that the image of $\phi$ is contained in that of $\psi$,
    and using Proposition~\ref{prop:C1asclassgroup} it follows
    that
    \begin{equation}
      \label{eq:inequal}
[\Mu(E):\Ml(E)] =
\frac{\#\Mu(E)}{\#\Cll(A_K,E)[2]} \le
\frac{\#\ker\phi}{\#(P/P_*(E))} \,.      
    \end{equation}

If $[\alpha]
 \in \ker\phi$, then $(\alpha) = (\beta)^2$, so
$\alpha = \beta^2 \mu$, with
$\mu \in A_\Om^\times$.
Thus
\[
    \ker\phi = \AsqO \cap \Mu(E). 
\]
The
sign map induces an isomorphism
\[
\frac{\AsqO}{\AsqO \cap \Mu(E)} \simeq   \frac{\sign(\AsqO)}{\sign(\AsqO)\cap \tilde{V}}.
\]
By the second isomorphism theorem,
\[
\frac{\sign(\AsqO)}{\sign(\AsqO)\cap \tilde{V}} \simeq
\frac{\sign(\AsqO) \cdot \tilde{V}}{\tilde{V}},
\]
hence
\[
    \frac{\# \AsqO}{\# \ker\phi}
    =
    \frac{\#(\sign(\AsqO) \cdot \tilde{V})}{\#\tilde{V}}
    =
    \frac{\#(\sign(A_\Om^\times) \cdot \tilde{V})}
         {\#\tilde{V}\,\#\sign(\Om^\times)},
\]
where the last equality follows from Lemma~\ref{lemma:directsum}.

On the other hand,
\[
    \frac{P}{P_*(E)} \simeq
    \frac{A_K^\times}{A_\Om^\times\,\cdot\,\sign^{-1}(\tilde{V})}
    \simeq
    \frac{\sign(A_K^\times)}{\sign(A_\Om^\times) \cdot \tilde{V}}
\]
via the $\sign$ map. We conclude
    \[
[\Mu(E):\Ml(E)] \le \frac{\#\ker\phi}{\#(P/P_*(E))}=
\frac{\#\tilde{V}\,\#\sign(\Om^\times)\,\# \AsqO}{\#\sign(A_K^\times)}
\]
and the theorem follows from 
$\#\tilde{V} = 2^b$, $\#\sign(\Om^\times) = 2^{a+b}$,
$\# \sign(A_K^\times) = 2^{a+3b}$,
$a+b+2c =[K:\QQ]$ and the following lemma.
\end{proof}

\begin{lemma}
  \label{lemma:Asqorder}
  With the previous notation, $\#\AsqO = 2^{a+2b+2c}$.
\end{lemma}
\begin{proof}
  Consider the norm map
  $\norm: A_\Om^\times/(A_\Om^\times)^2 \to
  \Om^\times/(\Om^\times)^2$. This map is surjective since $[A_K:K]=3$
  (given $\epsilon \in \Om^\times$, $\norm(\epsilon) = \epsilon$ up to
  squares) and $\AsqO$ is by definition its kernel.
  By Dirichlet's unit theorem we have
  $\# \Om^\times/(\Om^\times)^2=2^{a+b+c}$.
  Likewise we have
  $\#A_\Om^\times/(A_\Om^\times)^2=2^{2a+3b+3c}$,
  and the result follows.
\end{proof}

\begin{remark}
  The inequality in Theorem~\ref{thm:M2order} becomes an equality if
  the image of $\psi$ equals that of $\phi$; in that case, the
  inequality in (\ref{eq:inequal}) becomes an equality and the proof
  continues mutatis mutandis. This is the case if for example $K$ is a
  totally real number field.  The reason is that a totally positive
  number field $K$ with odd class number satisfies that all totally
  positive units are squares. Then if $I \in \Cll(A_K,E)[2]$, by
  definition $I^2 = (\alpha)$, with $\alpha \in P_*(E)$. Clearly
  $\alpha$ has even valuation at all finite places, and satisfies the
  hypothesis on elements of $\Mu(E)$ at the archimedean places by
  definition of $P_*(E)$. Note that $\norm(\alpha)$ is a square up to
  a unit (it matches the norm of $I^2$), and it is totally positive,
  hence the unit must be also a square.
\end{remark}

Combining
Proposition~\ref{prop:selmer inclusion},
Proposition~\ref{prop:C1asclassgroup}
and Theorem~\ref{thm:M2order}, we obtain
\begin{thm}
\label{thm:bound}
  Let $K$ be a number field and let $E/K$ be an elliptic curve satisfying hypotheses~\ref{hypo:KF}. Then
  \[
    \dim_{\FF_2}\Cll(A_K,E)[2] \le \dim_{\FF_2}\Sel_2(E) \le \dim_{\FF_2}\Cll(A_K,E)[2]+{[K:\Q]}.
  \]
  In particular, if $K=\QQ$, the order of the Selmer group is determined by the $2$-torsion of $\Cll(A_K,E)$ and the root number of $E$.
\end{thm}
This is a generalization of \cite[Theorem 2.18]{MR3934463}, noting
that if $\Disc(E)<0$ then $\Cll(A_\QQ,E) = \Cl(A_\QQ)$ (in particular it does not depend on the elliptic curve $E$). It is a natural
question whether the bound in Theorem~\ref{thm:bound} is sharp. We
will show some examples of elliptic curves over number fields which do
attain the lower and upper bound in Section~\ref{section:examples}.

\section{Application to quadratic twists}
For this section, $E/\QQ$ will denote an elliptic curve satisfying
hypotheses~\ref{hypo:KF}.  If $d \in \ZZ$, we denote by $E_d$ the
twist of $E$ by $\QQ(\sqrt{d})$, namely if $E$ is given by an equation
$E:y^2= F(x)$ then $E_d: dy^2 = F(x)$, which also equals
\[
  E_d:y^2=d^3F(x/d).
\]
Note that both $E$ and $E_d$ have the same attached cubic field.

\begin{lemma}
  If $d$ is a fundamental discriminant satisfying that all primes
  $p \mid d$ are inert or totally ramified in $A_\QQ$ then the twisted
  curve $E_d$ also satisfies hypotheses~\ref{hypo:KF}.
\end{lemma}

\begin{proof}
  By definition, we need to check the condition locally at each prime
  $p$. Clearly the condition (\dag.i) is invariant under twisting
  (since the attached cubic field is invariant). Note that all primes $p$
  dividing $d$ belong to the case (\dag.i) by the hypothesis.
  Consider now a prime $p\nmid d$.  If $E/\QQ_p$ satisfies (\dag.ii)
  then the discriminants of $F(x)$ and $d^3F(x/d)$ differ by a unit,
  hence $E_d$ also satisfies (\dag.ii).  If $E/\QQ_p$ satisfies
  (\dag.iii) then $E_d$ also satisfies (\dag.iii), since for each $p> 2$ the parity of $[E(\Q_p) : E_0(\Q_p)]$ and $[E_d(\Q_p) : (E_d)_0(\Q_p)]$
are equal (see the proof of \cite[Lemma 5.6]{MR3954912}). At last, if $E/\QQ_p$ satisfies (\dag.iv)
  then $E_d$ also satisfies (\dag.iv), because for
  $d\equiv 1\pmod{4}$, $E$ has good reduction at $2$ if and only if
  $E_d$ does.
\end{proof}

In particular, if $d$ is a fundamental discriminant such that all
primes $p \mid d$ are inert in $A_\QQ$, we can apply
Theorem~\ref{thm:bound} to both $E$ and $E_d$.
The caveat is that if $\Delta(E)>0$
the order of the roots of $F(T)$ (hence the distinguished place) is
reversed when $d<0$ and preserved when $d>0$, hence for $d<0$ the
class groups $\Cll(A_\QQ,E)$ and $\Cll(A_\QQ,E_d)$ might be
different (but only if $\Delta(E)>0$ and $\Cl_+(A_\QQ)\neq\Cl(A_\QQ)$).
This issue can be overcome if we consider pairs $d_1, d_2$
of discriminants satisfying the above hypothesis with $d_1/d_2>0$.

\begin{remark}
\label{rem:sameselmer}    
Let $E/\QQ$ be an elliptic curve satisfying hypotheses~\ref{hypo:KF},
and let $d_1, d_2$ be fundamental discriminants satisfying that all primes
$p \mid d_i$, $i=1,2$ are inert in $A_\QQ$.
Suppose also that either (a) $\Delta(E)<0$, (b)
$\Cl_+(A_\QQ)=\Cl(A_\QQ)$, or (c) $d_1/d_2>0$. Then we
have the following diagram
  \[
    \xymatrix{
      \Ml(E_{d_1}) \ar@{=}[d]  \ar@{}[r]|-*[@]{\subset} & \Sel_2(E_{d_1}) \ar@{}[r]|-*[@]{\subset} & \Mu(E_{d_1}) \ar@{=}[d] \\
      \Ml(E_{d_2}) \ar@{=}  \ar@{}[r]|-*[@]{\subset} & \Sel_2(E_{d_2}) \ar@{}[r]|-*[@]{\subset} & \Mu(E_{d_2})\\
    }
\]
As the index $[\Mu(E_{d_1}):\Ml(E_{d_1})]=2$,
then $\Sel_2(E_{d_1})=\Sel_2(E_{d_2})$ if and only if both curves have
the same root number. In particular we have infinitely many twists of
$E$ with the same $2$-Selmer group.
  \end{remark}

  We can explicitly determine for which discriminants both Selmer
  groups coincide, and say something on their densities if we restrict
  to prime discriminants.  Let $p$ be an odd prime number, and let
  $p^\ast = \kro{-1}{p}p$; recall that the quadratic extension of
  $\QQ$ unramified outside $p$ corresponds to
  $\QQ(\sqrt{p^\ast})$.
  Let $\epsilon(E)$ denote the root number of
  $E$. Recall that if $p \nmid 2\Disc(E)$, then
\begin{equation*}
    \epsilon(E) \epsilon(E_{p^\ast}) = \chi_p(-N_E),  
\end{equation*}
where $N_E$ is the conductor of $E$ and $\chi_p$ is the quadratic character unramified outside $p$.


\begin{thm}
\label{thm:negativetwists}
  Let $E/\QQ$ be an elliptic curve satisfying
  hypotheses~\ref{hypo:KF}, and suppose furthermore that
  either $\Disc(E)<0$ or $\Cl_+(A_\QQ)=\Cl(A_\QQ)$.
    Then the set of prime
  numbers $p$ inert in $A_\Q$ has density at least $1/3$ and for any
  such prime $p$ which does not divide $\Disc(E)$ it holds:
  \begin{itemize}
      \item if $-\frac{\Delta(E)}{N_E} \equiv \square \pmod{p}$
          then $E$ and $E_{p^\ast}$
      have the same root number.
      In particular, both curves have the same $2$-Selmer group,
    
  \item otherwise, $E$ and $E_{p^\ast}$ have opposite root number,
    and all curves $E_{p^\ast}$ in this second case have the same
    $2$-Selmer group.
  \end{itemize}
  In particular, the set of all quadratic twists of $E$ by prime discriminants has a subset of density at least $1/6$ where all curves in this set have the same $2$-Selmer group.
\end{thm}
\begin{proof}
The density of prime discriminants that are inert in $A_\QQ/\QQ$
 equals
  \begin{equation*}
    \text{density} =
    \begin{cases} \frac{2}{3} & \text{ if $A_\QQ/\QQ$ is Galois},\\
      \frac{1}{3} & \text{ otherwise.}
      \end{cases}
    \end{equation*}
 Recall that for an
  elliptic curve of the form $E:y^2=F(x)$,
  $\Disc(E) = 2^4\Disc(F(x))$, hence $\Disc(A_\Q)$ differ from
  $\Disc(E)$ by a square. In particular, since $p$ is
  inert in $A_\Q$, $\chi_p(\Disc(A_\Q))=1$, and 
  \begin{equation*}
      \epsilon(E) \epsilon(E_{p^\ast}) = \chi_p(-N_E) = \chi_p\left(-\frac{\Disc(E)}{N_E}\right).    
\end{equation*}
This proves the claim on the root numbers. The result on the
$2$-Selmer group follows from Remark~\ref{rem:sameselmer},
noting that when $\Delta(E)<0$ there are no real places of
type (ii) so in the bounds of Theorem~\ref{thm:bound}
are independent of $E$; and this is always the case when $\Cl_+(A_\QQ)=\Cl(A_\QQ)$.
  \end{proof}

  Recall the definitions given in the introduction: let $d_2(E)$
  denote the $2$-Selmer rank of $E$ and define
\[
  N_r(E,X) = |\{\text{quadratic } L/\QQ \; : \; d_2(E^L) = r \text{ and }|\delta(L/\QQ)|<X\}|,
  \]
  where $E^L$ denotes the quadratic twist of $E$ corresponding to $L$ and $\delta(L/\QQ)$ is the discriminant of the extension $L/\QQ$.
  
  \begin{coro}
    \label{coro:twists} Let $E/\QQ$ be an elliptic curve satisfying
    hypotheses~\ref{hypo:KF}, and suppose furthermore that either
    $\Disc(E)<0$ or $\Cl_+(A_\QQ)=\Cl(A_\QQ)$. Let $r \ge 0$, and
    suppose that $E$ has a quadratic twist by a prime inert in $A_\QQ$
    whose $2$-Selmer group has rank $r$.  Then
    $N_r(E,X) \gg X/\log(X)^{1-\alpha}$, where
    \[
      \alpha = \begin{cases}
        1/3 & \text{if }A_\QQ/\QQ \text{ is Galois},\\
        1/6 & \text{otherwise.}
      \end{cases}
      \]
  \end{coro}

  \begin{proof}
    The proof is a standard application of Ikehara's tauberian
    theorem, as explained in \cite{MR3954912}, proof of Theorem 1.12.
  \end{proof}
  
  \begin{remark} If $-\frac{\Disc(E)}{N_E}$ is a square then all inert
    primes lie in the first case of Theorem~\ref{thm:negativetwists}
    (and the proportion of twists with the same $2$-Selmer group
    raises to $1/3$, and in the previous Corollary the constant
    $\alpha$ is doubled). This is the case for example if the elliptic
    curve $E$ is semistable of odd conductor and
    $\Delta(E)<0$. 
    In such case Ogg's formula (\cite{MR952229}) implies that for each
    prime of (multiplicative) bad reduction the difference between the
    conductor and the discriminant valuations at an odd prime $p$
    equals the number of irreducible components of the Néron model
    minus one; we claim that hypotheses~\ref{hypo:KF} together with
    $E$ being semistable implies that such number is always odd, hence
    the result. Note that the $2$-division polynomial of a semistable
    curve always has a root on the base field, hence (\dag.i) cannot
    hold. The case (\dag.ii) implies that the discriminant of the
    polynomial has valuation $0$ or $1$ (recall that $p$ is odd),
    hence there is a unique component. Finally, the condition
    (\dag.iii) implies that the number of components is odd.

    A similar result holds for other elliptic curves where all primes
    of bad reduction satisfy that $[E(K):E_0(K)]$ is odd (i.e. condition 
    (\dag.iii) even for $p=2$), since for odd primes the hypothesis implies that the number of
    irreducible components in the N\'eron model of $E$ is odd and for $p=2$ the result follows from the proof of \cite[Lemma 5.9]{MR3954912}, end of part $(3)$.
  \end{remark}
Let $\C_1$
be the set of prime numbers which ramify completely or are totally
inert in $A_\Q$, and let
$K = \QQ\left(\sqrt{p^\ast}\; : \; p \in \C_1\right)$, an
infinite polyquadratic extension.
\begin{coro}
\label{coro:inifiteextension}
In the hypotheses of Theorem~\ref{thm:negativetwists},
suppose that $E$ has trivial $2$-Selmer
  group.
  Then $E(K)$ is finite.
\end{coro}
\begin{proof}
  If $P \in E(K)$ is a point of infinite order, then $P$ belongs to a
  finite polyquadratic subextension $L/\QQ$. Let $A = \Res^L_\QQ E$ be the restriction of scalars, so $E(L) = A(\QQ)$. There is an isogeny
  \[
      \phi: A \rightarrow \sum_{\chi} E_\chi,
\]
where $\chi$ runs over quadratic characters of $\Gal(L/\QQ)$. By
Theorems \ref{thm:bound} and ~\ref{thm:negativetwists} all curves $E_\chi$ have trivial
$2$-Selmer group, hence $P$ cannot have infinite order. We deduce the corollary by noting that $E(K)_{\text{tors}}$ is finite by \cite{Ribet}.
\end{proof}

  

\begin{exm}
  The elliptic curve $E_{11a1}$ with LMFDB label \lmfdbec{11}{a}{1}
  has no rational $2$-torsion points and is semistable. 
    Its cubic field corresponds to the polynomial $x^3-x^2+x+1$
  of discriminant $-44$. The prime $11$ is not totally ramified in
  $A_\QQ$, hence it does not belong to $\C_1$. The prime $2$ is
  totally ramified so $2\in\C_1$.  The set
  $\C_1 \subset \{p \; : \; \dkro{-44}{p}=1\}\cup\{2\}$, and over the
  polyquadratic extension $K=\QQ(\sqrt{p} \; : \; p \in \C_1)$, the
  group $E(K)$ is finite.

\end{exm}

For positive discriminants we get a similar result (with a
similar corollary); see also Example~\ref{exm:differentclassgroup}. Let $E/\QQ$ be an elliptic curve with $\Disc(E)>0$, and divide the set of primes inert in $A_\QQ$ into the following four different sets:
  \begin{itemize}
  \item $\C_{+,\square} = \{p \equiv 1 \pmod 4 \text{ such that }
    \frac{\Disc(E)}{N_E} \equiv \square \pmod{p}\}$,
  \item $\C_{+,\nsquare} = \{p \equiv 1 \pmod 4 \text{ such that }
    \frac{\Disc(E)}{N_E} \equiv \nsquare \pmod{p}\}$,
  \item $\C_{-,\square} = \{p \equiv 3 \pmod 4 \text{ such that }
    \frac{\Disc(E)}{N_E} \equiv \square \pmod{p}\}$,
  \item $\C_{-,\nsquare} = \{p \equiv 3 \pmod 4 \text{ such that }
    \frac{\Disc(E)}{N_E} \equiv \nsquare \pmod{p}\}$.
  \end{itemize}
The set $\C_{+,\square}$ is non-empty and has density at least $1/12$.
\begin{thm}
  \label{thm:positivetwists}
  Let $E/\QQ$ be an elliptic curve satisfying
  hypotheses~\ref{hypo:KF}, and suppose furthermore that $\Disc(E)>0$.
Then if $p$ is a prime inert in $A_\Q$ which does not divide
$\Disc(E)$, the root number of $E_{p^\ast}$ equals that of $E$ if
$p \in \C_{+,\square} \cup \C_{-,\square}$, while it is the opposite
one if $p \in \C_{+,\nsquare} \cup \C_{-,\nsquare}$. Furthermore, if
$p_1,p_2$ are inert primes in the same set,
$\Cll(A_\QQ,E_{p_1^\ast}) = \Cll(A_\QQ,E_{p_2^\ast})$.  In particular, if $p_1$
and $p_2$ belong to the same set, the curve $E_{p_1^\ast}$ and the
curve $E_{p_2^\ast}$ have the same $2$-Selmer group.
\end{thm}
\begin{proof}
  Note that primes in $\C_{+,\square} \cup \C_{+,\nsquare}$
  (i.e. $p\equiv 1 \pmod 4$) correspond to twists by real quadratic
  fields and primes in $\C_{-,\square} \cup \C_{-,\nsquare}$
  correspond to twists by imaginary quadratic fields.

  The proof mimics the negative discriminant case. To get the root
  number statement, note that
  $\chi_p(-N_E) = \chi_p(-1) \chi_p(\Disc(E))$. Then if
  $p\equiv 1 \pmod 4$, the same proof applies, while if
  $p \equiv 3 \pmod 4$, $\chi_p(-N_E) = -\chi_p(\Disc(E))$, which
  explains the change of root number.

  Regarding the $2$-Selmer statement, if $p_1$ and $p_2$ belong to the
  same set, the curves $E_{p_1^\ast}$ and $E_{p_2^\ast}$ are a
  positive quadratic twist of each other, hence
  $\Cll(A_\QQ,E_{p_1^\ast}) = \Cll(A_\QQ,E_{p_2^\ast})$ so the
  bound of Theorem~\ref{thm:bound} and Remark~\ref{rem:sameselmer} prove the statement.  
\end{proof}

An immediate application of the previous result is that when
$\Disc(E)>0$ among the set of all quadratic twists of $E$ there is a
subset with density at least $1/12$ satisfying that all curves on it
have the same $2$-Selmer group as $E$ (corresponding to the primes in
$\C_{+,\square}$). A result similar to Corollary~\ref{coro:inifiteextension} applies in this situation.

%
%
\subsection{General fields} The results of the previous section
have a natural analogue over a general number field $K$. Still there are
many subtleties, for example: it is not always true that given a prime
ideal $\id{p}$ of $K$ there is a quadratic extension of $K$ which is
unramified outside $\id{p}$ (and there might be more than one such
extension). The way to solve it is to consider quadratic extensions
$K(\sqrt{\alpha})/K$ of prime discriminant (instead of prime ideals),
and twist curves by them. Although most of the results for $\QQ$
extend mutatis mutandis for $K$, we give a weaker not technical
version.
  \begin{thm}
    Let $K$ be a number field and let $E/K$ be an elliptic curve
    satisfying hypotheses~\ref{hypo:KF}. Then among the quadratic
    twists of $E$ by quadratic extensions of prime discriminant, a
    positive proportion have $2$-Selmer group whose rank lies in the
    interval $[\Cll(A_K,E)[2],\Cll(A_K,E)[2]+[K:\QQ]]$.
\label{thm:quadtwist}
\end{thm}
\begin{proof}
  Considering only quadratic extensions $K(\sqrt{\alpha})$ of prime
  discriminant which are unramified at the archimedean places of $K$
  of type (ii), we can assure that the groups $\Cll(A_K,E)$ and
  $\Cll(A_K,E_{\alpha})$ are equal, hence the result follows from
  Theorem~\ref{thm:bound} and Remark~\ref{rem:sameselmer}.
\end{proof}

\begin{remark}
  A similar application of the previous Theorem gives a result in the
  spirit of Corollary~\ref{coro:twists} for general number
  fields. However, even if we fix the root number, we cannot state
  precisely which rank in the above interval is obtained infinitely
  many times (except for example when $[K:\QQ]=2$), hence our result is not as strong as that of \cite{MR2660452} (Theorem 1.4).
\end{remark}

    \section{Examples}
\label{section:examples}

The following examples were computed using
SageMath~\cite{sagemath} and PARI/GP~\cite{PARI2}.
The $2$-Selmer rank, when necessary, was computed using Magma~\cite{Magma}.

\subsection{Examples with \texorpdfstring{$K=\QQ$}{K = ℚ}}
\begin{exm}
Let $F(x) = x^3-x^2-54x+169$ (corresponding to
the elliptic curve \lmfdbec{106276}{a}{1}). Its rank equals $3$. The discriminant of $F(x)$
equals $163^2$, which also equals the discriminant of $A_\Q$, hence
(\dag.ii) is satisfied for all primes. Furthermore, since the
discriminant is a square, $A_\Q$ is a Galois extension of $\QQ$.  The
class group $\Cl(A_\Q) = \Cl_+(A_\Q) \simeq \ZZ/2 \times \ZZ/2$. In
particular, $\Cll(A_\Q,E_d) = \Cl(A_\Q)$ has $2$-rank $2$, hence
Theorems~\ref{thm:bound},~\ref{thm:negativetwists} and
\ref{thm:positivetwists} imply that the curve and all quadratic twists
by primes which are inert in $A_\Q$ have $2$-Selmer rank in
$\{2,3\}$.

In fact the sign of the functional equations
gives the parity of the $2$-Selmer rank
(see \cite[Theorem 1.5]{MR1381589}),
hence the $2$-Selmer rank of $E_p$
is $3$ for inert primes $p\equiv1\pmod{4}$ and
$2$ for inert primes $p\equiv3\pmod{4}$.
For instance, $E$ itself has $2$-Selmer rank $3$,
while its quadratic twist by $d=-3$ has $2$-Selmer rank $2$. In
particular, both bounds are attained.

If we consider twists by split primes
(not satisfying hypotheses~\ref{hypo:KF}) we check that the twists by
$d=-23$, $5$, $-347$, $241$, $-331$, $2341$
have $2$-Selmer rank $0$, $1$, $2$, $3$, $4$, $5$ (respectively),
so neither the lower or upper bounds hold.
\end{exm}

\begin{exm}
\label{exm:differentclassgroup}  
  Let $F(x)=x^3 - 7x + 3$ (corresponding to the elliptic curve
  \lmfdbec{9032}{a}{1}, see Example~\ref{exm:sortingroots}). Its rank equals $2$. The
  discriminant of $F(x)$ equals $1129$, which also equals the
  discriminant of $A_\Q$, hence (\dag.ii) is satisfied for all
  primes.  The class group $\Cl(A_\Q)$ is trivial but the
  narrow class group $\Cl_+(A_\Q)$ has order $2$.
  The ray class group $\Cll(A_\Q,E)$ also has order $2$.
  In particular, when taking quadratic twists by discriminants
  $d>0$ it turns out that $\Cll(A_\Q,E_d)=\Cll(A_\Q,E)$ has $2$-rank
  1, hence Theorem~\ref{thm:bound} implies that the curve and all
  quadratic twists by positive prime discriminants which are inert in
  $A_\Q$ have $2$-Selmer rank in $\{1,2\}$, determined by the sign of
  the functional equation. For instance,
  the quadratic twists by $d=5$ and $d=113$ have $2$-Selmer rank $1$
  and $2$, respectively.

  If we take quadratic
  twists by discriminants $d<0$, the distinguished real place changes, and
  $\Cll(A_\Q,E_d)$ is trivial, hence all quadratic twists by negative
  prime discriminants which are inert in $A_\Q$ have $2$-Selmer group
  rank in $\{0,1\}$, determined by the sign of the functional equation.
  For instance, the quadratic twists by $d=-43$ and $d=-7$ have
  $2$-Selmer rank $0$ and $1$, respectively.
\end{exm}

\subsection{Examples with \texorpdfstring{$K=\QQ(\sqrt{17})$}{K = ℚ(√17)}}
The quadratic field $K$ has trivial narrow
class group (hence it equals the class group).
\begin{exm}
    Let $F(x)=x^3+x+3$ (corresponding, over $\QQ$, to the elliptic
    curve \lmfdbec{1976}{a}{1}). Its rank equals $2$.
The
discriminant of $F(x)$ equals $-13\cdot19$, which also equals the
discriminant of $A_K$, hence (\dag.ii) is
satisfied for all primes.
The narrow class group of $A_K$ is
trivial, hence $\Cll(A_K,E_d)$ is trivial.
Theorem~\ref{thm:bound} thus
implies that the curve and all quadratic twists by primes which are
inert in $A_K$ have $2$-Selmer rank in $\{0,1,2\}$.

The curve itself,
and also the quadratic twist by $d=97+24\sqrt{17}$ of norm
$383$, have $2$-Selmer rank $2$, the quadratic twist by
$d=-13+2\sqrt{17}$ of norm $101$ has $2$-Selmer rank $1$, and the
quadratic twist by $d=45+8\sqrt{17}$ of norm $937$ has $2$-Selmer rank
$0$.  On the other hand the quadratic twist by $d=29+4\sqrt{17}$,
which is \emph{not} inert in $A_K$, has $2$-Selmer rank $3$.
\end{exm}




\subsection{Examples with $K = \text{\lmfdbnumberfield{3.1.23.1}}$} The field $K$
corresponds to the cubic field of discriminant $-23$ given by
$K=\QQ(\alpha)$ with $\alpha^3-\alpha^2+1$ and trivial narrow class group.
Since $[K:\QQ]=3$, our lower and upper bound in
Theorem~\ref{thm:bound} differ by $3$ so the
functional equation sign is not enough to determine the rank of the
$2$-Selmer group in any case. 

\begin{exm}
    Let $F(x)=x^3+x+3$ (corresponding, over $\QQ$, to the elliptic
    curve \lmfdbec{1976}{a}{1}).
The
discriminant of $F(x)$ equals $-13\cdot19$, which also equals the
discriminant of $A_K$, hence (\dag.ii) is
satisfied for all primes. Its rank equals $1$.

The narrow class group of $A_K$ is trivial, hence $\Cll(A_K,E)$ is trivial. 
Our bound implies that the curve and all quadratic twists by primes which are
inert in $A_K$ have $2$-Selmer rank in $\{0,1,2,3\}$.

The curve itself and the quadratic twist by
$-2\alpha^2 + \alpha - 2$ have $2$-Selmer rank $1$, and the quadratic
twist by $-4\alpha^2 + 3\alpha + 1$ has $2$-Selmer rank $0$. In
particular the lower bound is attained.

On the other hand, we note that all the quadratic twists by inert
prime discriminants of norm up to $100\,000$ (there are $808$ such
discriminants) have $2$-Selmer rank $0$ or $1$.
This is not explained by our results.
\end{exm}

\begin{exm}
    Let $F(x)=x^3+x+11$ (corresponding, over $\QQ$, to the elliptic
    curve \lmfdbec{26168}{a}{1}).
The
discriminant of $F(x)$ equals $-3271$, which also equals the
discriminant of $A_K$, hence (\dag.ii) is
satisfied for all primes. Its rank equals $4$.

The class group $\Cl(A_K)=\Cl_+(A_K)\simeq\ZZ/2$. In particular
$\Cll(A_K,E)=\Cl(A_K)$ has $2$-rank $1$.
Thus our bound implies that the curve and all quadratic twists by primes which are
inert in $A_K$ have $2$-Selmer rank in $\{1,2,3,4\}$.

The curve itself and the quadratic twist by
$-2\alpha^2 + \alpha - 2$ have $2$-Selmer rank $4$, and the quadratic
twist by $-\alpha^2 - \alpha + 4$ has $2$-Selmer rank $3$. In
particular the upper bound is attained.

On the other hand, we note that all the quadratic twists by inert
prime discriminants of norm up to $100\,000$ (there are $844$ such
discriminants) have $2$-Selmer rank $3$ or $4$.
This is not explained by our results.
\end{exm}

\section*{Acknowledgments} 
The present article grew as a research project in the BIRS-CMO
workshop ``Number Theory in the Americas'' held in Oaxaca. We would
like to thank the organizers for providing such a fruitful
environment. Part of this project was done in a visit of the second
author to the Centro de Matemática of Universidad de la República, we
want to thank the institution for its hospitality.  We thank Myungjun
Yu for reporting a mistake in an earlier verision of
Lemma~\ref{lemma:fixed}. Last but not least, we want to thank the
referee for many suggestions that improved the exposition and quality of the
present article.

\bibliographystyle{alpha}
\bibliography{biblio}

\end{document}